\documentclass[11pt]{article}
\usepackage[margin=1in]{geometry}
\usepackage{amsmath}
\usepackage{amssymb}
\usepackage{amsthm}
\usepackage{graphicx}
\usepackage{tikz-cd}
\usepackage{multicol}
\usepackage{setspace}
\usepackage[utf8]{inputenc}
\usepackage[backend=biber,style=authoryear,sorting=nyt,language=english]{biblatex}
\newtheorem{teo}{Theorem}
\newtheorem{obs}{Remark}
\newtheorem{defi}{Definition}

\title{An Algebraic Structure for the Central Mexican Ritual Calendar}
\author{Ramiro Carrillo Catal\''an \thanks{Dedicated to Professor Isa\''ias Aldaz Hern\''andez (\textit{Tsyay}).}\ \small{Universidad Pedag\''ogica Nacional Oaxaca -- Secihti.}\ \small{\texttt{rcarrillo@g.upn.mx}}}
\date{March 12, 2026\ \small{\textit{Ce Cipactli, Ce Tochtli}}}
\begin{document}
\maketitle
  \begin{abstract}
To keep track of time, Mesoamerican peoples made use of two calendars: the 260-day ritual calendar and the 365-day solar calendar. The structure of the 260-day ritual calendar system, the \textit{Tonalpohualli} (the count of days in the Nahuatl language), has intrigued and attracted scholarly interest from the colonial period \parencite{duran1880,veytia1836} to the present day. Multiple efforts have been made to interpret and understand its configuration \parencite{caso1967,milbrath2013,seler1979} and, in particular, its mathematical structure. In \textcite{pinagarza1999} a first arithmetic description is proposed, and \textcite[p.~46]{dehouve2023} concludes that this calendrical system possesses a closed structure (in the mathematical sense) upon studying the structural relations among the positions of this system. The present work offers, through an \textit{etic} analysis, a description of an algebraic structure that models the Central Mexican calendrical system. This structure confirms and complements the findings of \parencite{dehouve2023,pinagarza1999} by showing that the 260-day ritual calendar, the \textit{Tonalpohualli}, can be described using the cyclic group $\mathbb{Z}_{13}\oplus\mathbb{Z}_{20}$, and that it is possible to determine the relationship between the day number and the corresponding day name through an action of this group. In a similar fashion, the twenty \textit{trecenas} and the thirteen \textit{veintenas} of the ritual calendar are described algebraically as orbits of two transformations. Also, using transformations of the group $\mathbb{Z}_{13}\oplus\mathbb{Z}_{20}$, the grouping of four days (tetrads) is described both for all 260 days and for the first day of each \textit{trecena}. Finally, the relationship between these groupings and the subgroup generated by the twenty \textit{trecenas} of the \textit{Tonalpohualli}, expressed as permutations, is established. These descriptions may prove useful for the interpretation of codices.
    \end{abstract}

\noindent \textbf{Keywords:} \textit{Tonalpohualli}, modular arithmetic, cyclic groups, group actions.\\

\section{The Central Mexican Ritual Calendar}

To keep track of time, Mesoamerican peoples made use of two calendars: the 260-day ritual calendar and the 365-day solar calendar. There is evidence that the use of these calendrical systems was widespread throughout Mesoamerica at the time of the Spanish Conquest \parencite{milbrath2013,rodriguezcano2017,rojasmartinezgracida2022}. In the 260-day ritual calendar system, known in Nahuatl as the \textit{Tonalpohualli} (the count of days), a system of twenty calendrical signs combined with numerals from 1 to 13 was used to assign names to the days \parencite{caso1967}. The name of each day was determined by two components: a numeral and a calendrical sign (see Table~1, appended). The $260 = 13 \times 20$ different combinations are obtained by cyclically pairing the 13 integers with the 20 signs, which repeat in periods of 260 days. Thus, the sequence of day names proceeds as follows: 1-Crocodile, 2-Wind, 3-House, \ldots, 13-Reed, 1-Jaguar, 2-Eagle, \ldots, 7-Flower, 8-Crocodile, \ldots, 13-Flower (see Table~2, appended). The use of this calendrical system can be observed in the 10 surviving pre-Columbian codices: both in the 5 Mixtec codices and in the 5 codices of the so-called Borgia group\footnote{The representation of the Tonalpohualli in codices and books is usually termed \textit{Tonal\'{a}matl} (paper of the days).}. For example, in the first 8 plates of the Borgia Codex, the 260 days of the \textit{Tonalpohualli} can be identified with the 260 signs representing twenty thirteen-day periods (\textit{trecenas}), arranged in 4 groups of 65 days, each ordered as an array of 5 rows and 13 columns \parencite{andersjansen1993,dehouve2023,diazrodgersbyland1993,milbrath2013}.

In the codices \textit{Tonindeye} (Zouche-Nuttall) and \textit{Yuta Tnoho} (Vindobonensis), numerous dates and names of persons represented through this system appear frequently \parencite{andersjansenperezjimenez1992a,andersjansenperezjimenez1992b,nuttallmiller1975}. On the other hand, the 365-day system known as the \textit{Xiuhpohualli} consisted of 18 groups of 20 days each---called \textit{veintenas} by the Spaniards---and five additional days called \textit{Nemontemi} \parencite{milbrath2013}. Each \textit{veintena} corresponded to a specific ritual celebration of 20 days' duration, and the additional days were regarded as days \textit{that complete what has been lived}.

The recording of time was accomplished by considering simultaneously the 260-day ritual calendar and the 365-day annual cycle, both incorporated into a period of 52 years corresponding to 73 cycles of the \textit{Tonalpohualli} ($73 \times 260 = 52 \times 365$). This period is known as the \textit{Xiuhmolpilli} (binding of years). Under this system, a day recurs with the same name after 18,980 days, that is, after 52 years, which is the least common multiple of 260 and 365. The names of each year were given by combinations of an integer from 1 to 13 and a sign called a year-bearer. Of the twenty calendrical signs, only four were used as year-bearers: Reed, Flint, House, and Rabbit \parencite{caso1967,dehouve2023,milbrath2013}. There exist early colonial books and codices that describe this manner of organizing temporal periods, such as the Chronological Tables and the Calendrical Wheel No.~5 by Veytia, which shows the year divided into \textit{veintenas} \parencite{veytia1836}, or the temporal wheel of Dur\'{a}n, where the \textit{Xiuhmolpilli} is represented in spiral form, organizing the 52 years into four oriented groups \parencite{duran1880}.

In what follows, we provide an algebraic description of the ritual calendrical system that permits us, from an \textit{etic} perspective, to appreciate the complexity of the Mesoamerican system of timekeeping. The study of the relationship between the ritual calendar \textit{Tonalpohualli} and the 365-day solar calendar \textit{Xiuhmolpilli} will be addressed in a separate work.

\section{The \textit{Tonalpohualli} as a Cyclic Group}

As mentioned above, the name of each of the 260 days of the \textit{Tonalpohualli} is formed by the $13 \times 20$ different combinations obtained by cyclically pairing the 13 integers with the 20 calendrical signs. We note that the integers from 1 to 13 repeat throughout the combinations with each of the 20 signs and, conversely, the 20 signs repeat as they are matched with the 13 numerals (see Table~2). Accordingly, these sets can be modeled as the additive group of integers modulo 13 and modulo 20, respectively. Thus, the name of each day can be represented as an ordered pair, and since these names repeat in cycles of 260 days, the complete calendrical system can be modeled as the direct sum of those groups.\\

\begin{obs}
The \textit{Tonalpohualli} can be modeled as the group $\mathbb{Z}_{13}\oplus\mathbb{Z}_{20}$.
\end{obs}

To see this, let us first identify the set of the 13 numerals with the set of integers from 0 to 12, and the set of calendrical signs with the integers from 0 to 19, where the last sign, Flower, corresponds to 0. Then
$\Upsilon=\{ (q,r) \mid a \in \{0, \ldots, 12\},\, b\in \{0, \ldots, 19\} \}$ with the operation given by:
$$(q_{1},r_{1})\oplus (q_{2},r_{2})= (q_{1}+q_{2} \bmod{13}, \quad r_{1}+r_{2} \bmod{20})$$
coincides exactly with the definition of the group $\mathbb{Z}_{13}\oplus\mathbb{Z}_{20}$. We observe that since the orders 13 and 20 of the component groups are relatively prime, the direct sum $\mathbb{Z}_{13}\oplus\mathbb{Z}_{20}$ is a cyclic group of order 260 with generator $(1,1)$.\\

\subsection{The Number and Name of the Days}

We will now show how the above model allows us, given the number of a day from 1 to 260, to determine the numeral and the calendrical sign that form its name in the ritual calendar, and conversely, given the day name in the ritual calendar, to determine the corresponding day number. To this end, we consider $x$ as an element of the set $\{1,2,\ldots,260\}$ and $(q,r)$ as an element of $\mathbb{Z}_{13}\oplus\mathbb{Z}_{20}$. However, since days in a calendrical system repeat in cycles, in what follows we will consider $x\in \mathbb{Z}_{260}$; that is, $x$ represents the day number between 0 and 259, and $(q, r)$ represents its calendrical name.

First, given a value of $x$, we seek to determine the values of $q$ and $r$. In other words, we seek to establish the correspondence:
$$\ell:\mathbb{Z}_{260} \longmapsto \mathbb{Z}_{13}\oplus\mathbb{Z}_{20}$$
where each $x$ corresponds to some $(q,r)$. Since $q \in \mathbb{Z}_{13}$ and $r \in \mathbb{Z}_{20}$, we must have $x\equiv q \mod 13$ and $x\equiv r \mod 20$. This is equivalent to solving $x= 13k+q$ and $x= 20m+r$ for integers $k$ and $m$. Thus the map $\ell$ is defined by
$$\ell(x):= ( x\equiv q \bmod 13,\; x\equiv r \bmod 20).$$

\textbf{Example:} Consider day number 197 of 260. To determine the numeral and calendrical sign forming its name, we find $q$ and $r$ satisfying $197 \equiv q \bmod 13$ and $197 \equiv r \bmod 20$. Since $13(15)= 195$, we obtain $q = 2$; and since $197 = 20(9)+17$, we obtain $r=17$. Thus the name assigned to day 197 is $(2, 17)$, or (2, Movement).\\

Now suppose we know the day name in the ritual calendar and wish to determine the corresponding day number. In this case, the values of $q$ and $r$ forming the element $(q,r)$ of $\mathbb{Z}_{13}\oplus\mathbb{Z}_{20}$ are given, and we need to determine $x\in\mathbb{Z}_{260}$. In other words, we seek to establish the correspondence:
$$\imath:\mathbb{Z}_{13}\oplus\mathbb{Z}_{20} \longmapsto \mathbb{Z}_{260}$$
that assigns to the name of each day its corresponding day number. Under this correspondence, day (1, Crocodile) or (1,1), being the first day of the ritual calendar, should correspond to day 1; day (2, Wind) or (2,2) to day 2; day (7, Flower) or (7,0) to day 20; and day (8, Crocodile) or (8,1) to day 21. The last day of the calendar, day 260, must correspond to (13, Flower) or (13,0).\\

Given a pair $(q,r)$, we must simultaneously solve the congruences $x\equiv q \mod 13$ and $x\equiv r \mod 20$. This is a classical problem solvable by the Chinese Remainder Theorem, which establishes that a unique solution modulo the product of the moduli exists when the moduli are relatively prime---which is precisely the case here, since the moduli are 13 and 20. Thus $\imath((q,r))=x$, where $x$ is the simultaneous solution of $x\equiv q \mod 13$ and $x\equiv r \mod 20$, is well-defined.

From $x\equiv q \mod{13}$ we write $x=13k+q$ for some integer $k$. Substituting into the second congruence gives $13k+q\equiv r \mod 20$, equivalently $13k\equiv (r-q) \mod 20$, so we must solve for $k$. To determine $k$, we need an integer $y$ such that $13y \equiv 1 \mod 20$. Testing multiples of 13 shows that $13\times 17=221\equiv 1\pmod{20}$, so the inverse of 13 modulo 20 is 17. Thus we solve $13k\equiv (r-q) \mod 20$ by multiplying both sides by 17: $17(13k)\equiv 17(r-q) \bmod 20$.
Since $13\times 17\equiv 1\mod 20$, it follows that $k\equiv 17(r-q) \mod 20$, so $x=13k+q$ with $k\equiv 17(r-q) \bmod 20$.

On the other hand, the Chinese Remainder Theorem gives the unique solution explicitly as:
$$x \equiv 20qy_{1} + 13ry_{2} \bmod 260$$
where $y_1$ is the inverse of 20 modulo 13 (i.e., $20y_1 \equiv 1 \bmod 13$) and $y_2$ is the inverse of 13 modulo 20 (i.e., $13y_{2} \equiv 1 \bmod 20$). As established, $y_2=17$; and $y_{1}=2$ since $40=13(3)+1$. Therefore:
$$x \equiv 40q + 221r\bmod 260.$$
Since $221=260-39$, we have $221r\equiv -39r \bmod 260$, so $x \equiv 40q -39r\bmod 260$.

We conclude that given $(q,r)\in\mathbb{Z}_{13}\oplus\mathbb{Z}_{20}$, its image $\imath((q,r))=x$ is given by
$$x \equiv (40q -39r) \bmod 260.$$
One may verify that $\imath((1,1))=1$, $\imath((2,2))=2$, $\imath((7,0))=20$, $\imath((8,1))=21$, and $\imath((13,0))=0$, as required.\\

\newpage
\textbf{Example:} Consider the day (4, Deer) or (4,7). To find its day number, we solve simultaneously $x\equiv 4 \mod{13}$ and $x\equiv 7 \mod{20}$. Writing $x=13k+4$ and substituting: $13k+4\equiv 7 \mod 20$, so $13k\equiv 3 \mod 20$. Multiplying by the inverse 17: $(17)(13k)\equiv 51 \mod 20$, hence $k\equiv 11 \mod 20$. Setting $k=20n+11$:
$$ x = 13(20n+11)+4= 260n + 147.$$
Thus the day number corresponding to day (4, Deer) within one cycle of the \textit{Tonalpohualli} is day~\textbf{147}.\\

\begin{teo}
We have that $\mathbb{Z}_{13}\oplus \mathbb{Z}_{20}\cong\mathbb{Z}_{260}$ via the isomorphisms $\ell$ and $\imath$. Moreover, $\ell$ and $\imath$ are inverse mappings of each other.
\end{teo}
\begin{proof}
The mapping $\imath$ is an isomorphism: it is \textit{injective} since if two elements $x_1, x_2\in\mathbb{Z}_{260}$ are both images of the same pair $(q, r)$---i.e., solutions of the same system of congruences---then their difference $x_1 - x_2$ is divisible by both 13 and 20. Since 13 and 20 are relatively prime, $x_1 - x_2$ must be divisible by 260, so $x_1 = x_2$ in $\mathbb{Z}_{260}$. Thus each pair $(q, r)$ corresponds to a unique $x$. It is also \textit{surjective}: the Chinese Remainder Theorem guarantees that a solution $x$ exists and is unique modulo 260; since the set of possible values of $x$ has exactly 260 elements, every $x$ is the image of some pair.

Furthermore, $\imath$ preserves addition in $\mathbb{Z}_{13}\oplus \mathbb{Z}_{20}$, since for $(q_1+q_2, r_1+r_2)\in \mathbb{Z}_{260}$ we have $\imath ((q_1+q_2, r_1+r_2)) \equiv x_1+x_2 \mod 260$.

Now $\ell$ is also an isomorphism: it is injective since if $\ell(x_1)=\ell(x_2)$, then $x_1\equiv x_2 \bmod 13$ and $x_1\equiv x_2 \bmod 20$; since 13 and 20 are relatively prime, $x_1\equiv x_2 \bmod 260$, i.e., $x_1=x_2$ in $\mathbb{Z}_{260}$. It is surjective since given any $(q,r)\in\mathbb{Z}_{13}\oplus \mathbb{Z}_{20}$, the Chinese Remainder Theorem guarantees the existence of a unique $x\in\mathbb{Z}_{260}$ such that $x\equiv q \bmod 13$ and $x\equiv r \bmod 20$ simultaneously, so that $\ell(x)=(q,r)$; this $x$ is precisely $\imath((q,r))$.

$\ell$ preserves the addition operation from $\mathbb{Z}_{260}$ to $\mathbb{Z}_{13}\oplus \mathbb{Z}_{20}$. For $x_1, x_2\in\mathbb{Z}_{260}$:
$\ell(x_1+x_2)=((x_1+x_2)\bmod 13, (x_1+x_2)\bmod 20)$, and:
\begin{align*}
\ell(x_1)+ \ell(x_2)&=(x_1\bmod 13, x_1\bmod 20)+(x_2\bmod 13, x_2\bmod 20)\\
&=((x_1 \bmod 13) + (x_2 \bmod 13) \bmod 13,\; (x_1 \bmod 20)+(x_2\bmod 20) \bmod 20).
\end{align*}
Since $(x \bmod m) + (y\bmod m)\equiv x+y \bmod m$ for any modulus $m$, we have
$(x_1\bmod 13) + (x_2\bmod 13)\equiv (x_1+x_2)\bmod 13$
and similarly for modulo 20. Therefore $\ell(x_1+x_2)=\ell(x_1)+\ell(x_2)$ in $\mathbb{Z}_{13}\oplus \mathbb{Z}_{20}$.

Finally, these mappings are inverses of each other. For any $(q,r)\in \mathbb{Z}_{13}\oplus \mathbb{Z}_{20}$, $\imath((q,r))=x$ is the unique element satisfying both congruences; applying $\ell$:
$\ell(\imath((q,r)))=(x\bmod13, x\bmod 20) = (q,r)$, so $\ell\circ\imath=\mathrm{Id}_{\mathbb{Z}_{13}\oplus \mathbb{Z}_{20}}$. Conversely, for $x\in\mathbb{Z}_{260}$, $\ell(x)=(x\bmod 13, x\bmod 20)$, and $\imath$ applied to this pair gives the unique solution: $\imath(\ell(x))=x$, so $\imath\circ\ell=\textrm{Id}_{\mathbb{Z}_{260}}$.
\end{proof}

\begin{defi}
We call the system $\tau=\langle \mathbb{Z}_{13}\oplus\mathbb{Z}_{20}, \ell, \imath \rangle$ the \textbf{base \textit{Tonalpohualli}}.
\end{defi}

The foregoing results can be interpreted as follows:
\begin{obs}
For $\tau$, the base \textit{Tonalpohualli}:
\begin{itemize}
\item Given $x$, a day numbered from 1 to 260, the numeral $q$ and the calendrical sign $r$ constituting its calendrical name are determined by $\ell(x)=(x\equiv q \bmod 13,\; x\equiv r \bmod 20).$
\item Given $(q,r)$, the calendrical name of a day, the corresponding day number $x$ is given by: $x=\imath((q,r))= 13k + q$ with $k\equiv 17(r-q) \bmod 20$, or equivalently $x\equiv (40q-39r)\bmod 260$.
\end{itemize}
\end{obs}

\newpage

\subsection{The \textit{Tonalpohualli} as a permutations group} 

We have observed that identifying the structure of the \textit{Tonalpohualli} with the direct sum $\mathbb{Z}_{13}\oplus\mathbb{Z}_{20}$ allowed us to formulate, in an organized manner, the correspondences between the day names of the ritual calendar and the 260 days identified with $\mathbb{Z}_{260}$.

For the study of the \textit{Tonalpohualli}, it is of particular interest to identify the day names that mark the beginning of each \textit{trecena} or each \textit{veintena}. Specifically, given the calendrical name of an arbitrary day, one wishes to identify and determine the name of the starting day of its corresponding \textit{trecena} or \textit{veintena}. As shown in this section, this can be accomplished by considering permutations of the set $\{0, 1,2,3,\ldots, 259\}$ or, more generally, through actions of $\mathbb{Z}_{13}\oplus\mathbb{Z}_{20}$ on that set.

Cayley's theorem, a classical result in algebra, states that every finite group can be represented as a group of permutations. We now show how $\mathbb{Z}_{13}\oplus\mathbb{Z}_{20}$ can be embedded in $S_{260}$, the symmetric group of order 260---that is, the group of all possible permutations of a set of 260 elements. Recall that each permutation is a bijective function from the set to itself, and the group operation is composition of permutations.

To this end, we construct a correspondence that transforms each element of $\mathbb{Z}_{13}\oplus\mathbb{Z}_{20}$ into a permutation of $\{0, 1,2,3,\ldots, 259\}$. In this section we work with this set rather than the full group $\mathbb{Z}_{260}$.

Let $(a,b)$ be a fixed element of $\mathbb{Z}_{13}\oplus\mathbb{Z}_{20}$. Define the transformation $\varphi_{(a,b)}$ by:
\begin{align*}
\varphi_{(a,b)}:\mathbb{Z}_{13}\oplus\mathbb{Z}_{20} &\longmapsto \mathbb{Z}_{13}\oplus\mathbb{Z}_{20} \\
 (x,y)&\longmapsto (a+x \mod 13,\; b+y \mod 20).
\end{align*}
This transformation defines an action of $\mathbb{Z}_{13}\oplus\mathbb{Z}_{20}$ on itself---the left-translation action---since addition modulo 13 and 20 is associative, and for the identity element $(0,0)$ we have $\varphi_{(0,0)}((x,y))=(x,y)$.

For each $(a,b)\in \mathbb{Z}_{13}\oplus\mathbb{Z}_{20}$, the translation $\varphi_{(a,b)}$ induces a permutation of $\{0,1,2,\ldots, 259\}$: under $\ell$, each day number corresponds to a day name in the ritual calendar, and under $\imath$, each day name corresponds to a day number.\\

The composition $\imath\circ\varphi_{(a,b)}\circ\ell$ is therefore a permutation of $\{0,1,2,\ldots, 259\}$:
$$\begin{tikzcd}[column sep=large, row sep=large]
\mathbb{Z}_{13}\oplus\mathbb{Z}_{20}
  \arrow[rr, "\varphi_{(a,b)}"]
&&
\mathbb{Z}_{13}\oplus\mathbb{Z}_{20}
  \arrow[dl, "\imath"]
\\
& \{0,\ldots,259\} \arrow[ul, "\ell"'] &
\end{tikzcd}$$
$\imath\circ\varphi_{(a,b)}\circ\ell(x)=\imath(\varphi_{(a,b)}(\ell(x)))=\imath(\varphi_{(a,b)}(q,r))=\imath\!\left( (a+q \bmod 13,\; b+r \bmod 20) \right)=y$,

where $y=14k+(a+q)$ with $k\equiv 17[(b+r)-(a+q)] \bmod 20$, so $x$ is mapped to $y$. Since this holds for every $x$ and any $(a,b)$, the composition is indeed a permutation.

We observe that the generator $(1,1)$ of $\mathbb{Z}_{13}\oplus\mathbb{Z}_{20}$ corresponds to the permutation $(x, y) \mapsto (x+1 \mod13,\; y+1 \mod 20)$, which traverses all 260 elements. Thus, through $\varphi_{(a,b)}$, the group $\mathbb{Z}_{13}\oplus\mathbb{Z}_{20}$ is expressed as a group of permutations.

This composition induces an action of $\mathbb{Z}_{13}\oplus\mathbb{Z}_{20}$ on the set $\{0,1,2,\ldots, 259\}$:
\begin{align*}
\mathrm{T}:\mathbb{Z} _{13}\oplus\mathbb{Z}_{20}\times \{0,1,2,\ldots, 259\} &\longmapsto \{0,1,2,\ldots, 259\} \\
((a,b), x) &\longmapsto y
\end{align*}
given by $\mathrm{T}_{(a,b)}(x)=\imath\circ\varphi_{(a,b)}\circ\ell(x)$. The identity element acts trivially and the action preserves the group operation:
\begin{itemize}
\item For $(0,0)\in \mathbb{Z} _{13}\oplus\mathbb{Z}_{20}$, $\varphi_{(0,0)}$ is the identity, so $\mathrm{T}((0,0),x)=\imath(\varphi_{(0,0)}(\ell(x)))=\imath(\ell(x))=x$.
\item Since addition is associative in $\mathbb{Z} _{13}\oplus\mathbb{Z}_{20}$, we have $\mathrm{T}_{(a+c, b+d)}(x)=\imath(\varphi_{(a+c, b+d)}(\ell(x)))$.
On the other hand, if $\mathrm{T}_{(c,d)}(x)=\imath(\varphi_{(c,d)}(\ell(x))):=y$ and $\mathrm{T}_{(a,b)}(y)=\imath(\varphi_{(a,b)}(\ell(y)))$, then $\ell(y)=\varphi_{(c,d)}(\ell(x))$ since $\imath$ and $\ell$ are inverses. Therefore:
$$\mathrm{T}_{(a,b)}(y)=\imath(\varphi_{(a,b)}(\varphi_{(c,d)}(\ell(x))))= \imath((\varphi_{(a,b)}\circ\varphi_{(c,d)})(\ell(x))).$$
Since $\varphi_{(a,b)}\circ \varphi_{(c,d)}=\varphi_{(a+c,b+d)}$, we conclude:
$$\mathrm{T}_{(a,b)}(\mathrm{T}_{(c,d)}(x))=\imath(\varphi_{(a+c,b+d)}(\ell(x)))= \mathrm{T}_{(a+c,b+d)}(x).$$
\end{itemize}

We observe that if, instead of the set $\{0, 1,2,\ldots, 259\}$, we considered the full group $\mathbb{Z}_{260}$, the action would be an inner automorphism of $\mathbb{Z}_{260}$, corresponding to translation by the fixed element $t=\imath((a,b))\in \mathbb{Z}_{260}$. Thus $\mathbb{Z}_{260}\cong\mathbb{Z}_{13}\oplus\mathbb{Z}_{20}$ is embedded in $S_{260}$.

\begin{teo}
The map $\mathrm{T}_{(a,b)}(x)=\imath\circ\varphi_{(a,b)}\circ\ell(x)$ is an action of the group $\mathbb{Z}_{13}\oplus\mathbb{Z}_{20}$ on the set $\{1,2,3,\ldots, 259,0\}$ that embeds $\mathbb{Z}_{13}\oplus\mathbb{Z}_{20}$ into $S_{260}$. Each element of $\mathbb{Z}_{13}\oplus\mathbb{Z}_{20}$ is associated with a permutation of $\{1,2,3,\ldots, 259,0\}$.
\end{teo}

With the foregoing, any permutation of the set $\{0, 1,2,\ldots, 259\}$ induced by the group $\mathbb{Z}_{13}\oplus\mathbb{Z}_{20}$ can be explicitly constructed. Moreover, each $\varphi_{(a,b)}$ may be regarded as a transformation that translates the days of the ritual calendar as required.

\textbf{Example:} To move day number 1 to day number 21---that is, day (1, Crocodile) or (1,1) to day (8, Crocodile) or (8,1)---we seek $(a,b)\in\mathbb{Z}_{13}\oplus\mathbb{Z}_{20}$ satisfying $\varphi_{(a,b)}(1,1)=(8,1)$.

By definition of $\varphi_{(a,b)}$: $(a+1 \mod 13,\; b+1 \mod 20)=(8,1)$, so $a+1\equiv 8 \mod 13$ and $b+1 \equiv 1 \mod 20$. Thus $a \equiv 7 \mod 13$ (take $a=7$) and $b\equiv 0 \mod 20$ (take $b=0$). Hence $\varphi_{(7,0)}(1,1)=(8,1)$, and $\varphi_{(7,0)}$ moves day 1 to day 21.

We also note that $\varphi_{(7,0)}(8,1)=(7+8 \mod 13,\; 0+1 \mod 20)= (2,1)$, so $\varphi_{(7,0)}$ moves day 21 to day 41. Furthermore, $\varphi_{(7,0)}(10,10)=(4,10)$ and $\varphi_{(7,0)}(7,0)=(1,0)$, so day 10 moves to day 30 and day 20 to day 40. In general, $\varphi_{(7,0)}$ corresponds to a shift of 20 days.

Further examples: $\varphi_{(1,0)}$ shifts by 40 days, $\varphi_{(8,0)}$ shifts by 60 days, and $\varphi_{(6,0)}$ shifts by 240 days. Three transformations of this type are of particular interest for the study of codices: the one identifying the start of each \textit{trecena}, the one identifying the start of each \textit{veintena}, and the one forming tetrads.

\subsection{The Twenty \textit{Trecenas} and the Thirteen \textit{Veintenas} of the \textit{Tonalpohualli}}

As mentioned earlier, the days of the \textit{Tonalpohualli} are organized into twenty \textit{trecenas}. We now wish to determine a transformation that moves days between \textit{trecenas}. Following the same procedure as before, we seek an element $(a, b)$ of $\mathbb{Z}_{13}\oplus\mathbb{Z}_{20}$ such that the associated transformation $\varphi_{(a,b)}$ sends day (1,1) to (1,14). The condition $\varphi_{(a,b)}(1,1)=(1,14)$ requires:
$$\varphi_{(a,b)}(1,1)=(a+1 \mod 13\;,\; b+1 \mod 20)=(1,14)$$
so $a+1\equiv 1 \mod 13$ and $b+1\equiv 14 \mod 20$, equivalently $a \equiv 0 \mod 13$ and $b\equiv 13 \mod 20$. We may take $a=0$ and $b=13$.

The transformation $\varphi_{(0,13)}$ indeed moves days between \textit{trecenas}: $\varphi_{(0,13)}(1,1)= (1,14)$, $\varphi_{(0,13)}(1,14)= (1,7)$, $\varphi_{(0,13)}(1,7)= (1, 0)$, and $\varphi_{(0,13)}(1,8)= (1, 1)$. That is, day (1, Crocodile) or (1,1) moves to (1, Jaguar) or (1,14); successively applying $\varphi_{(0,13)}$, (1, Jaguar) moves to (1, Deer), which moves to (1, Flower); finally, (1, Rabbit) moves back to (1, Crocodile). Thus we can determine the numeral and calendrical sign forming the name of the first day of each thirteen-day period. Note that in each case the numeral is always 1; what varies is the calendrical sign.

The orbit of the element (1,1) under $\varphi_{(0,13)}$ is:
$$\varphi_{(0,13)}^{n}(1,1)=\left\{ \; \big(1, \; 1+13n \mod 20 \big) \in \mathbb{Z}_{13}\oplus\mathbb{Z}_{20}, \; n\in\mathbb{N} \right\}$$
where
$$\varphi^{n}_{(0,13)} =\varphi_{(0,13)}\!\left(\varphi_{(0,13)}\!\left(\cdots \varphi_{(0,13)}\!\left(\varphi_{(0,13)}(1,1)\cdots \right)\right)\right) \quad n\text{-times.}$$
This set represents the $n$-th iteration of $\varphi_{(0,13)}$ applied to (1,1).

In Table~3, iterating the transformation yields elements of $\mathbb{Z}_{13}\oplus\mathbb{Z}_{20}$ of the form $(1,r)$; the different values of $r$ give the list of the 20 starting signs of each \textit{trecena}. Equivalently, this is the restriction of the orbit of (1,1) under $\varphi_{(0,13)}$ to $\mathbb{Z}_{20}$:
$$\varphi_{(0,13)}^{n}(1,1)\big|_{\mathbb{Z}_{20}}=\left\{  r\in \mathbb{Z}_{20} \;\mid\; r= 1+ 13n \mod 20 \right\}.$$
Considering the orbit of each element $(q,r)$ of $\mathbb{Z}_{13}\oplus\mathbb{Z}_{20}$ yields the calendrical signs of the second day of each \textit{trecena}, the third day, and so on. This allows us to construct Table~2, organized by the start of the \textit{trecenas}, which is equivalent to Table~1.1 in \parencite[p.~4]{milbrath2013}.

We may proceed analogously with $\varphi_{(7,0)}$---the 20-day shift constructed in the previous section---whose orbit of (1,1) is:
$$\varphi_{(7,0)}^{n}(1,1)=\left\{ \; \big(1+7n \mod 13, \; 1 \big) \in \mathbb{Z}_{13}\oplus\mathbb{Z}_{20}, \; n\in\mathbb{N} \right\}.$$
The restriction of this orbit to $\mathbb{Z}_{13}$:
$$\varphi_{(7,0)}^{n}(1,1)\big|_{\mathbb{Z}_{13}}=\left\{ q\in \mathbb{Z}_{13} \;\mid\; q= 1+ 7n \mod 13 \right\}$$
gives the list of numerals of the day names that begin each 20-day period.

Considering the orbit of each element $(q,r)$ of $\mathbb{Z}_{13}\oplus\mathbb{Z}_{20}$ under $\varphi_{(7,0)}$ yields sets equivalent to Table~1 in \parencite[p.~33]{dehouve2023}. In Table~3, the elements of $\mathbb{Z}_{13}\oplus\mathbb{Z}_{20}$ of the form $(q,1)$ (marked in grey) give the 13 numerals of the names of the days that begin each \textit{veintena}.

\begin{teo}
In $\mathbb{Z}_{13}\oplus\mathbb{Z}_{20}$, the transformation $\varphi_{(0,13)}$ corresponds to a shift of 13 elements and the transformation $\varphi_{(7,0)}$ corresponds to a shift of 20 elements. The restrictions of the orbits of $(1,1)$ in the \textit{Tonalpohualli}:
\begin{itemize}
\item $\varphi_{(0,13)}^{n}(1,1)\big|_{\mathbb{Z}_{20}}$ corresponds to the set of calendrical signs of the day names that begin each \textit{trecena}.
\item $\varphi_{(7,0)}^{n}(1,1)\big|_{\mathbb{Z}_{13}}$ corresponds to the list of numerals of the day names that begin each \textit{veintena}.
\end{itemize}
\end{teo}

Therefore, given a \textit{veintena}, one can determine the numeral corresponding to the name of the day that begins it, and vice versa. Likewise, given a \textit{trecena}, one can determine the calendrical sign of the day name with which it begins, and vice versa.

\subsection{The Tetrads of the \textit{Tonalpohualli}}

In plates 1 through 8 of the Borgia Codex, the 260 calendrical signs are arranged in four groups of two plates, each pair presenting 65 signs organized in 5 rows and 13 columns. The name corresponding to each day is implicit: the numeral is inferred from the position of the calendrical sign in the array. This constitutes an in-extenso representation of the 260 days of the \textit{Tonalpohualli} \parencite{andersjansen1993,diazrodgersbyland1993,milbrath2013,seler1979}. The days are read from right to left starting at the lower-right corner of plate 1, continuing until the first 52 days are complete at the end of the first row in plate 8. The sequence then continues at the beginning of the second row in plate 1, proceeding until the end of that row in plate 8 (completing 104 days), and so on. The first column of each pair of plates presents the 5 different signs associated with the names of the days that begin each \textit{trecena}.

Both \parencite{seler1979} and \parencite{dehouve2023} emphasize the various ways of dividing the \textit{Tonalpohualli} into tetrads---sets of four days or four \textit{trecenas}: the 260 days can be grouped into 4 sets of 65, into sets of 4 consecutive days, and each 52-day block can be subdivided into four \textit{trecenas}. Moreover, as noted in the first section, since the 52 years of the \textit{Xiuhmolpilli} can be organized into four groups each associated with a cardinal direction \parencite{duran1880}, Seler and Dehouve speak, by analogy, of both oriented days and oriented \textit{trecenas}, assigning each day or \textit{trecena} an element of the ordered set \{East, North, West, South\}.

We now show how these tetrad subdivisions can be described in terms of transformations of the group $\mathbb{Z}_{13}\oplus\mathbb{Z}_{20}$. In particular, we wish to determine the set of calendrical signs corresponding to the names of each oriented \textit{trecena}, as well as the numerals and signs of the day names grouped by orientation.

To this end, we determine the transformation that shifts each day by 4 days: we require that (1, Crocodile) move to (5, Serpent), i.e., (1,1) maps to (5,5). We seek $(a,b)\in\mathbb{Z}_{13}\oplus\mathbb{Z}_{20}$ such that
$$\varphi_{(a,b)}{(1,1)}=(a+1\mod 13\;,\;b+1\mod 20) = (5,5)$$
so $a\equiv 4 \mod 13$ and $b\equiv 4 \mod 20$; we take $a=4,\, b=4$, giving
$$\varphi_{(4,4)}(x,y)=(4+x \mod 13\;,\; 4+y \mod 20).$$
Under this transformation, day 1 goes to 5, day 2 to 6, day 3 to 7, and so on---a permutation of all 260 elements of $\mathbb{Z}_{13}\oplus\mathbb{Z}_{20}$, as before. The orbit of (1,1) under $\varphi_{(4,4)}$ is:
$$\varphi_{(4,4)}^{n}(1,1)=\left\{ \; \big(1+4n \mod 13, \; 1+4n \mod 20 \big) \in \mathbb{Z}_{13}\oplus\mathbb{Z}_{20}, \; n\in\mathbb{N} \right\}.$$
Table~4 presents the elements of this orbit after 11 iterations. We observe that the calendrical signs of each day name repeat every 5 iterations, and if we associate the orientation \textit{East} with day (1,1), following \parencite{dehouve2023}, then this orbit corresponds to the names of all days associated with that orientation.

That is, the orbit $\varphi^{n}_{(4,4)}(1,1)$ corresponds to the days associated with the East. The orbits $\varphi^{n}_{(4,4)}(2,2)$, $\varphi^{n}_{(4,4)}(3,3)$, and $\varphi^{n}_{(4,4)}(4,4)$ correspond to the day names associated with the North, West, and South, respectively. This is illustrated in Table~4 for the orbits of (1,1) and (4,4) under this transformation.
\newpage

To determine the names of the first days of each oriented \textit{trecena}---specifically the signs of those days, as represented in the first column of plates 1, 3, 5, and 7 of the Borgia Codex---we hold the day numeral fixed and consider the effects of the transformation on the second coordinate only, for those day names that correspond to the start of each \textit{trecena}. This amounts to considering the orbits of (1,1), (1,14), (1,7), and (1,0) under $\varphi_{(0,4)}$, respectively.

Table~5 shows the values of the orbits $\varphi_{(0,4)}^n(1,1)$, $\varphi_{(0,4)}^n(1,14)$, $\varphi_{(0,4)}^n(1,7)$, and $\varphi_{(0,4)}^n(1,0)$. The sets of calendrical signs associated with the start of each oriented \textit{trecena} are: $\{1, 5, 9, 13, 17\}$ (East), $\{14, 18, 2, 6, 10\}$ (North), $\{7, 11, 15, 19, 3\}$ (West), and $\{0,4,8,12,16\}$ (South). These results coincide with the findings presented in Tables~3 and~4 in \parencite[pp.~37--38]{dehouve2023}. We now discuss the relationship between these sets and the twenty \textit{trecenas} of the \textit{Tonalpohualli}.

\subsection{The Twenty \textit{Trecenas} of the \textit{Tonalpohualli} as Permutations and Their Relationship with the Tetrad Decomposition}

We now observe that the restriction $\varphi_{(0,13)}^{n}(1,1)\big|_{\mathbb{Z}_{20}}$ allows us to express the list of the 20 calendrical signs marking the start of each thirteen-day period as a permutation of the 20 calendrical signs, by an argument analogous to that in Section~2.2: Cayley's theorem allows us to express this restriction as a permutation in the symmetric group $S_{20}$. Each of the twenty \textit{trecenas} is assigned the calendrical sign of the name of the day with which it begins. For example, \textit{trecena}~5 is assigned the value 13, corresponding to the sign Reed, since the name of the day beginning that \textit{trecena} is (1,13) or (1, Reed). As another example, the fourth \textit{trecena} is assigned the value 0, which modulo 20 corresponds to the sign Flower, since it begins with day (1,20) or (1, Flower). Thus the restriction $\varphi_{(0,13)}^{n}(1,1)\big|_{\mathbb{Z}_{20}}$ of the complete orbit, in permutation notation, is:
$$\sigma = \left(\begin{array}{cccccccccccccccccccc} 1 & 2 & 3 & 4& 5& 6& 7& 8& 9& 10& 11& 12& 13& 14& 15& 16& 17& 18& 19& 0\\
 1 & 14 & 7 & 0& 13& 6& 19& 12& 5& 18& 11& 4& 17& 10& 3& 16& 9& 2& 15& 8 \end{array}\right)$$
This is none other than the permutation presented by Pi\~{n}a Garza, who in 1999 had already outlined some of the basic arithmetic properties that hint at the underlying algebraic structure \parencite{pinagarza1999}. In cycle notation:
$$ \sigma = (1) \; (2,14,10,18) \; (3,7,19,15) \; (4,0,8,12)  \; (5, 13,17,9)  \; (6) \; (11) \; (16)$$
We observe that $\sigma$ has 4 fixed points---1, 6, 11, and 16---and can be expressed as a product of 4 disjoint cycles. Since these cycles are disjoint, the order of $\sigma$ equals the least common multiple of the cycle lengths, which is 4.

Moreover, since the sign of a permutation is determined by the number of transpositions into which it decomposes, and each 4-cycle decomposes into 3 transpositions, the total number of transpositions for the four cycles is 12; hence the permutation is even. The decomposition into cycles is:
$$ \sigma =  (2,14,10,18) \; (3,7,19,15) \; (4,0,8,12)  \; (5, 13,17,9).$$

Considering the powers of $\sigma$ (its products with itself under composition):
\begin{itemize}
\item$\sigma^{0} = \textit{id}$
\item$\sigma^{1} =\sigma$
\item$\sigma^{2}=(2,10)(14,18)(3,19)(7,15)(4,8)(20,12)(5,17)(13,9)$
\item$\sigma^{3}=(2,18,10,14)(3,15,19,7)(4,12,8,20)(5,9,17,13)$
\end{itemize}
We observe that $\sigma^{3}=\sigma^{-1}$, and $(\sigma^{-1})^{2}=\sigma^{2}$. This set satisfies the group axioms (closure, identity, inverses), so $\langle\sigma\rangle$ is the cyclic subgroup of $S_{20}$ generated by $\sigma$:
$$\langle\sigma\rangle = \{ \sigma^{0} = \textit{id}, \; \sigma^{1} =\sigma, \; \sigma^{2}, \; \sigma^{3}\}.$$
This subgroup is isomorphic to the finite cyclic group of order 4, $\mathbb{Z}/4\mathbb{Z} = \{ \overline{0}, \overline{1}, \overline{2}, \overline{3}\}$ with addition modulo 4: $\overline{a}+\overline{b} = \overline{a+b} \mod 4$, where $\overline{k}$ denotes the equivalence class of integers congruent to $k$ modulo 4.
\newpage
Assigning each equivalence class modulo 4 to the corresponding power of $\sigma$, we construct a map $\phi: \mathbb{Z}/4\mathbb{Z} \longrightarrow \langle\sigma\rangle$ defined by $\phi(\overline{k})=\sigma^{k}$. This map preserves the group operation: for any $\overline{a},\overline{b}\in \mathbb{Z}/4\mathbb{Z}$,
$$\phi(\overline{a}+\overline{b})=\sigma^{a+b}=\sigma^{a}\circ\sigma^{b}=\phi(\overline{a})\circ\phi(\overline{b})$$
since exponentiation in a cyclic group is additive. Injectivity follows from $\phi(\overline{a})=\phi(\overline{b})\Rightarrow \sigma^{a}=\sigma^{b}\Rightarrow a\equiv b\mod 4\Rightarrow \overline{a}=\overline{b}$; surjectivity is clear. Thus $\phi$ is an isomorphism.

Finally, $\mathbb{Z}/4\mathbb{Z}$ is also isomorphic to the rotation group of the square, $\mathrm{Rot}(\square)$---the group of four rotations (counterclockwise by 0°, 90°, 180°, and 270°) under composition. The isomorphism $\psi: \mathbb{Z}/4\mathbb{Z} \longrightarrow \mathrm{Rot}(\square)$ is given by $\psi(\overline{k}) = \text{rotation by } 90^{\circ}\times k$, where $k\in\{0,1,2,3\}$. Since addition of angles corresponds to composition of rotations, $\psi$ is indeed an isomorphism.

This coincides with the cultural conclusions established by Dehouve in \parencite{dehouve2023}: the groupings into tetrads and their association with the cardinal directions and their rotations are an inherent feature of the algebraic structure of the Mesoamerican ritual calendar.

\newpage
\section{Conclusions and Perspectives}

The \textit{Tonalpohualli}, the 260-day ritual calendar widely used throughout Mesoamerica, can be described as the group $\mathbb{Z}_{13}\oplus\mathbb{Z}_{20}$. This representation allows us, given a day number, to determine the numeral and calendrical sign forming its name in the ritual calendar, and conversely, given the day name, to determine the corresponding day number. Moreover, by Cayley's theorem, this group can be viewed as a subgroup of permutations of the symmetric group $S_{260}$ via the action $\mathrm{T}_{(a,b)}(x)=\imath\circ\varphi_{(a,b)}\circ\ell(x)$ of $\mathbb{Z}_{13}\oplus\mathbb{Z}_{20}$ on $\{1,2,\ldots,259,0\}$.

From this action, the transformations $\varphi_{(7, 0)}$ and $\varphi_{(0,13)}$---which shift each day by 20 and 13 days, respectively---allow us, through the orbits $\varphi_{(0,13)}^{n}(1,1)\big|_{\mathbb{Z}_{20}}$ and $\varphi_{(7,0)}^{n}(1,1)\big|_{\mathbb{Z}_{13}}$, to determine: given a \textit{veintena}, the numeral of the day name with which it begins, and vice versa; and given a \textit{trecena}, the calendrical sign of the day name with which it begins, and vice versa. Analogously, the transformation $\varphi_{(4,4)}$ applied to the first four days organizes the remaining 260 days into groups of 4 and associates them with the corresponding orientations. The orbits $\varphi_{(0,4)}^n(1,1)$, $\varphi_{(0,4)}^n(1,14)$, $\varphi_{(0,4)}^n(1,7)$, and $\varphi_{(0,4)}^n(1,0)$ correspond to the sets of signs of the day names beginning each oriented \textit{trecena}. Furthermore, the list of the 20 calendrical signs marking the start of each thirteen-day period can be expressed, via the orbit $\varphi_{(0,13)}^{n}(1,1)\big|_{\mathbb{Z}_{20}}$, as the permutation
$$ \sigma =  (2,14,10,18) \; (3,7,19,15) \; (4,0,8,12)  \; (5, 13,17,9)$$
of the 20 calendrical signs. The cycle decomposition of this permutation coincides with the sets of starting signs of each oriented \textit{trecena}. The subgroup of $S_{20}$ generated by this permutation is isomorphic to $\mathbb{Z}/4\mathbb{Z}$, which is in turn isomorphic to the rotation group of the square (0°, 90°, 180°, 270°)---a result consistent with the correspondence of each day or \textit{trecena} with an oriented tetrad (East, North, West, South).

These representations allow us to determine, precisely, the relationships between the day number, the numeral and calendrical sign forming its name, and its position by \textit{trecena}, \textit{veintena}, or tetrad. This in turn aids in the reading of codices---for example, the identification of \textit{trecenas} and day names in plates 47 and 48 of the Borgia Codex, or the various dates in the codex \textit{Yuta Tnoho} (Vindobonensis) or the codex \textit{Tonindeye} (Zouche-Nuttall).

The algebraic description discussed in this work, far from exhaustive, reveals that much remains to be studied: for instance, the relationship of the group $\mathbb{Z}_{13}\oplus\mathbb{Z}_{20}$ or the subgroup $\mathbb{Z}/4\mathbb{Z}$ with the division of the year into 18 festivities of 20 days each, or with the regularity of footprint symbols in the Borgia Codex, as well as the relationship of these groups with the calendrical ratio of $146{:}104$ \textit{tlalcuahuitl}\footnote{A \textit{tlalcuahuitl} was a unit of length used by the Mexica, equivalent to 2.5 meters \parencite{jorgejorgewilliams2011}. Thus 365 meters and 260 meters correspond to 146 and 104 \textit{tlalcuahuitl}, respectively.} reflected in the architectural designs of Teotihuac\'{a}n and Monte Alb\'{a}n \parencite{peelerwinter2011}.

Finally, understanding the algebraic structure of the ritual calendar through the lens of contemporary mathematics enables us to reflect on the complexity and systematicity of the measurement and conception of time in pre-Columbian Mesoamerica, and in turn to reflect on the strengthening of our identity, the reappropriation, and the revaluation of the knowledge traditions of the diverse cultures that constitute our heritage.

\newpage
\section*{Appendices}

\begin{center}
\includegraphics[width=0.92\textwidth]{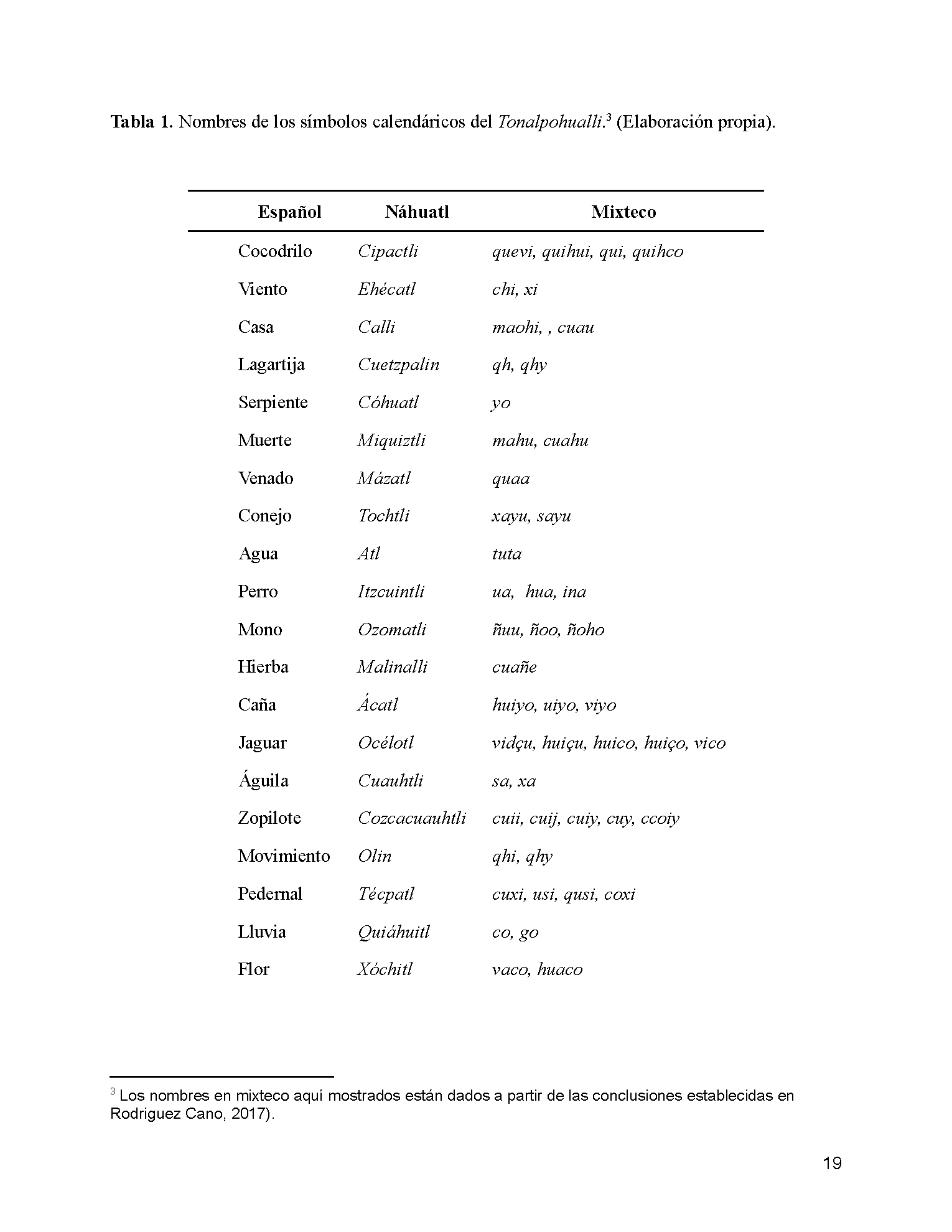}
\end{center}

\begin{figure}[p]
\centering
\includegraphics[height=0.94\textheight,keepaspectratio]{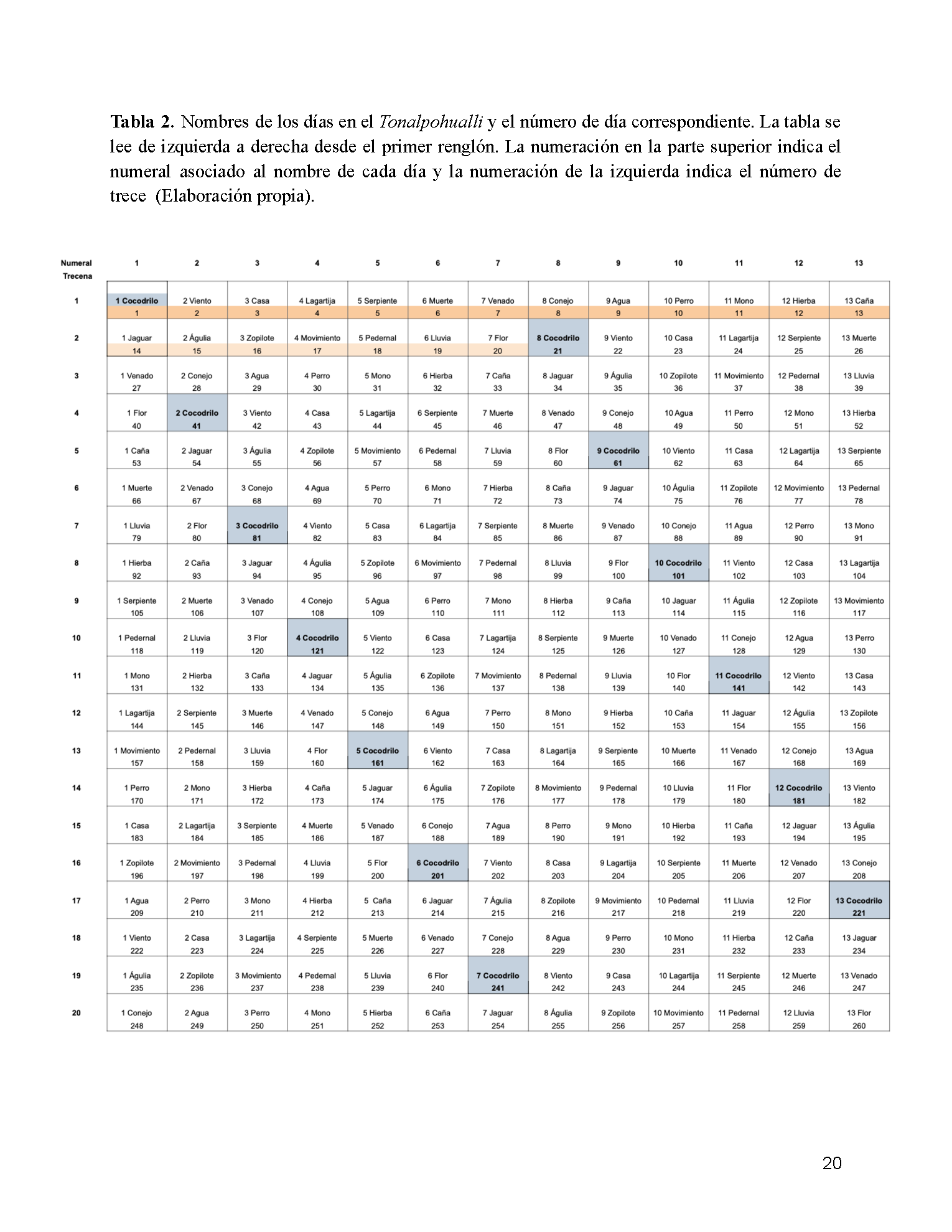}
\end{figure}

\begin{figure}[p]
\centering
\includegraphics[width=0.92\textwidth]{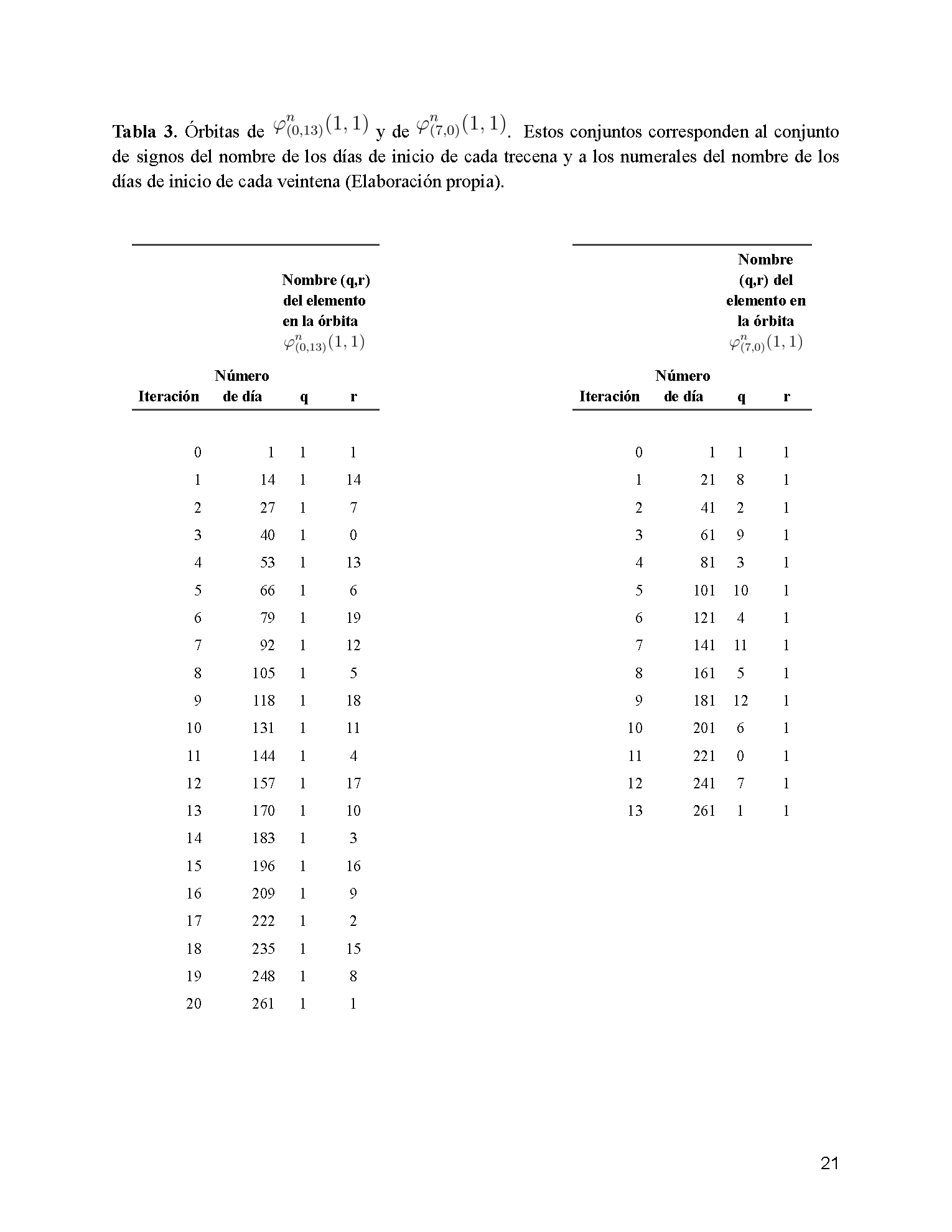}
\end{figure}

\begin{figure}[p]
\centering
\includegraphics[width=0.92\textwidth]{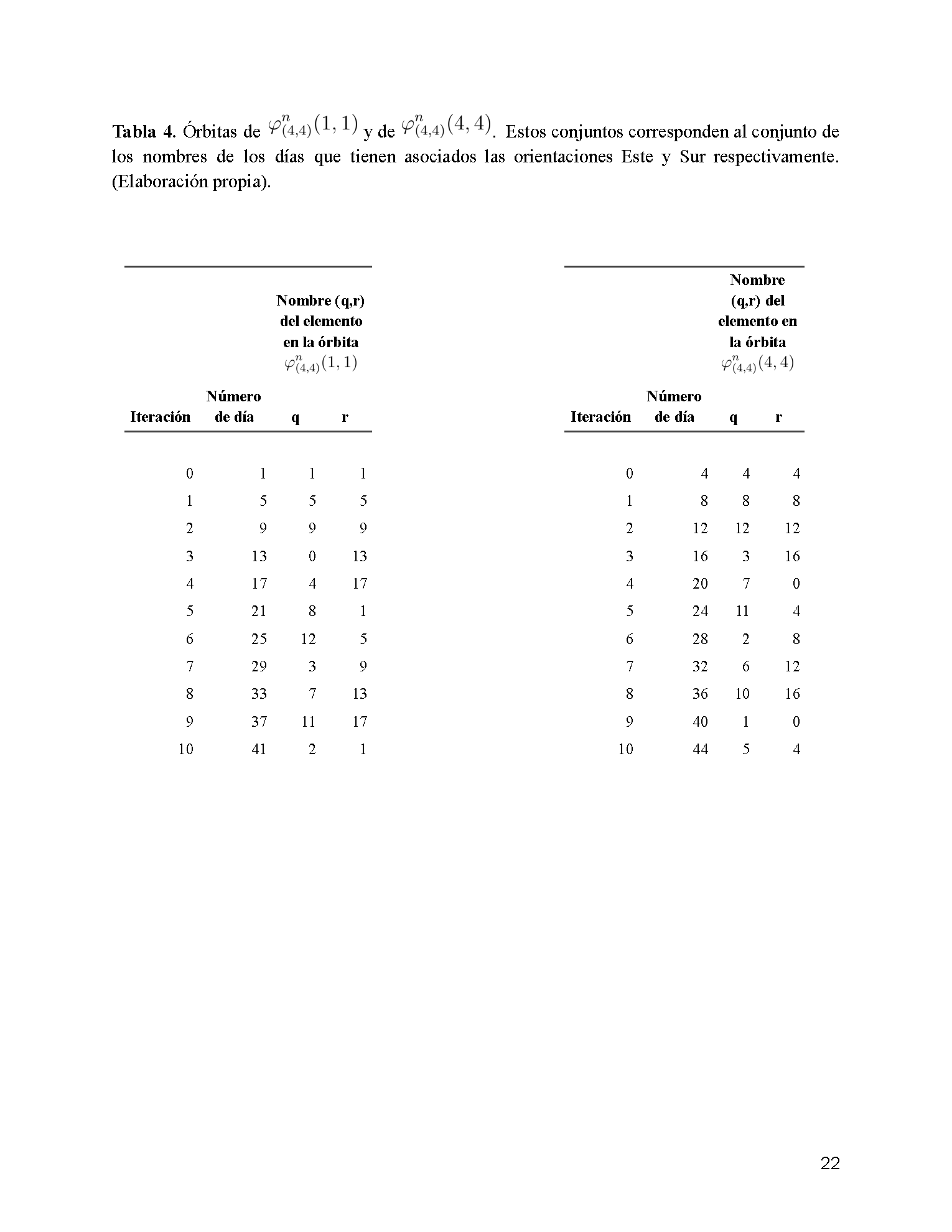}
\end{figure}

\begin{figure}[p]
\centering
\includegraphics[width=0.92\textwidth]{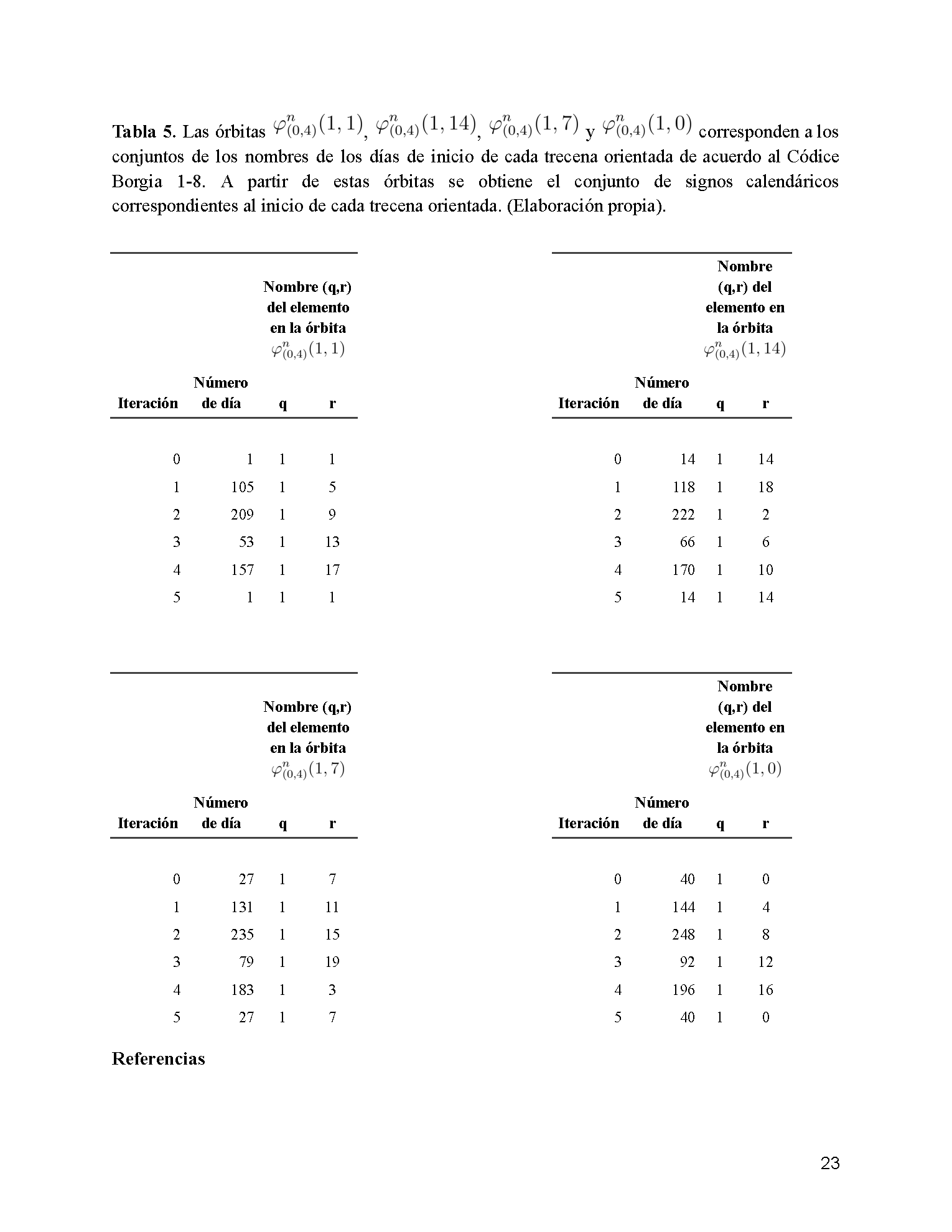}
\end{figure}

\clearpage
\printbibliography[title={References}]

\end{document}